\documentclass[reqno]{amsart}
\usepackage{verbatim}
\usepackage{amssymb}
\usepackage{enumerate}
\usepackage[active]{srcltx}
\usepackage[T1]{fontenc}

\numberwithin{equation}{section}

\newtheorem{theorem}{Theorem}[section]
\newtheorem{proposition}[theorem]{Proposition}
\newtheorem{lemma}[theorem]{Lemma}
\newtheorem{corollary}[theorem]{Corollary}

\newtheorem{problem}{Problem}

\theoremstyle{definition}

\newtheorem{definition}[theorem]{Definition}

\theoremstyle{remark}

\def\myheads#1;#2;{
\pagestyle{myheadings}
\markboth{{\sc\hfill #1\hfill\protect\makebox[0cm][r]{\rm\today}}}
{{\sc\protect\makebox[0cm][l]{\rm\today}\hfill #2\hfill}}
}

\newif\ifdeveloping

\ifdeveloping
\usepackage[notref,notcite]{showkeys}
\fi

\newif\ifcommented

\newcommand{\comm}[1]{}

\ifcommented
\renewcommand{\comm}[1]{
\fbox{\fbox{\begin{minipage}{300pt}#1\end{minipage}}
}}

\fi

\newcommand{\setm}{\setminus}

\newcommand{\subs}{\subset}
\newcommand{\dom}{\operatorname{dom}}

\def\<{\left\langle}
\def\>{\right\rangle}
\def\br#1;#2;{\bigl[ {#1} \bigr]^ {#2} }
\newcommand{\cf}{\operatorname{cf}}

\newcommand{\pr}{\operatorname{Pr}}
\newcommand{\Col}{\operatorname{Col}}

\newcommand{\supp}{\operatorname{supp}}

\author[I. Juhász]{István Juhász}
\address
      {Alfréd Rényi Institute of Mathematics, Hungarian Academy of Sciences}
\email{juhasz@renyi.hu}

\author[S. Shelah]{Saharon Shelah}
\address{Einstein Institute of Mathematics, The Hebrew University of Jerusalem}
\email{shelah@math.huji.ac.il}

\title[discretely untouchable points]{Strong colorings yield $\kappa$-bounded spaces with discretely untouchable points}

\thanks{The first author was partially supported by OTKA grant no. K 83726
and thanks the Mathematics Institute of the Hebrew University for
their hospitality during his visit there in May, 2013. The second
author would like to thank the Israel Science Foundation for
partial support of this research, Grant no. 1053/ll. Publication
1025 on the author's list}

\date{\today}

\begin{document}


\begin{abstract}
It is well-known that every non-isolated point in a compact
Hausdorff space is the accumulation point of a discrete subset.
Answering a question raised by Z. Szentmikl\'ossy and the first
author, we show that this statement fails for countably compact
regular spaces, and even for $\omega$-bounded regular spaces. In
fact, there are $\kappa$-bounded counterexamples for every
infinite cardinal $\kappa$. The proof makes essential use of the
so-called {\em strong colorings} that were invented by the second
author.
\end{abstract}

\maketitle

\section{Introduction}\label{intro}

It is part of topology folklore that a topological space is
compact iff any discrete subset in it has compact closure. Since
compact subsets of Hausdorff spaces are closed, it follows that
every non-isolated point in a compact Hausdorff space is the
accumulation point of a discrete subset, or in other words, each
such point is "discretely touchable".

Motivated by this fact, Z. Szentmikl\'ossy and the first author
raised the natural question whether this property of non-isolated
points in compact Hausdorff spaces remains valid after relaxing
compactness to a weaker property, like countable compactness or
Lindel\"ofness. The aim of this note is to give a negative answer
to this question and in fact show that no essential relaxation of
compactness suffices to preserve the above statement.

Before turning to the proof of this, we present a few preliminary
results concerning discretely touchable points. First of all we
note that any accumulation point of a right separated (or
equivalently: scattered) subspace is discretely touchable. Indeed,
this follows from the fact that the set of isolated points of any
right separated space, which is clearly discrete, is dense in that
space. This simple observation yields us the following
proposition.

\begin{proposition}\label{pr:lim}
If $X$ is a Hausdorff space, $\kappa$ is an infinite cardinal, and
the point $x$ is the limit of a one-to-one $\kappa$-sequence in
$X$ then $x$ is discretely touchable in $X$.
\end{proposition}

\begin{proof}
Clearly, we may assume that $\kappa$ is regular, hence if $s =
\langle x_\alpha : \alpha < \kappa \rangle$ is the one-to-one
$\kappa$-sequence converging to $x$ then any intersection of fewer
than $\kappa$ many neighbourhoods of $x$ contains a tail of the
sequence $s$. As $X$ is Hausdorff, the sigleton $\{x\}$ is the
intersection of the closed neighbourhoods of the point $x$, hence
by a straight-forward transfinite induction we may select a
cofinal subsequence $\langle x_{\alpha_\nu} : \nu < \kappa
\rangle$ of $s$ such that $\{ x_\alpha : \alpha \ge \alpha_\nu \}$
is closed in $\{ x_\alpha : \alpha < \kappa \}$ for each $\nu <
\kappa$. But then $\{ x_{\alpha_\nu} : \nu < \kappa \}$ is clearly
a right separated subset of $X$ that accumulates (even converges)
to the point $x$.

\end{proof}

\begin{corollary}\label{co:chi=psi}
If we have $\chi(p,X)= \psi(p,X) \ge \omega$ for the point $p$ in
the Hausdorff space $X$ then $p$ is discretely touchable in $X$.
\end{corollary}

\begin{proof}
It is straight-forward to show that if $\chi(p,X)= \psi(p,X) =
\kappa \ge \omega$ then there is one-to-one $\kappa$-sequence in
$X$ converging to $p$.

\end{proof}

Since $\chi(p,X)= \psi(p,X)$ for each point $p$ of a compact
Hausdorff space $X$, corollary \ref{co:chi=psi} yields an
alternative way of showing our starting point which was the fact
that every non-isolated point in a compact Hausdorff space is is
discretely touchable.

This also leads us to the following result that perhaps explains
why it seems to be non-trivial to find a discretely untouchable
non-isolated point in a regular countably compact space.

\begin{corollary}\label{co:w_2}
If $x$ is a discretely untouchable non-isolated point in a regular
countably compact space $X$ then we have $$\omega < \psi(x,X) <
\chi(x,X)\,.$$ In particular, then $\chi(x,X) \ge \omega_2\,$.
\end{corollary}

\begin{proof}
Indeed, it is well-known that if $p \in X$ and $\psi(p,X) =
\omega$ in a regular countably compact space $X$ then we also have
$\chi(p,X) = \omega$.

\end{proof}

A completely similar argument as above, using the fact that any
point $x$ in an initially $\kappa$-compact regular space $X$ with
pseudo-character $\psi(x,X) \le \kappa$ satisfies $\chi(x,X) =
\psi(x,X),\,$ yields the following more general result.

\begin{proposition}\label{pr:init}
If $x$ is a discretely untouchable non-isolated point in a regular
initially $\kappa$-compact space $X$ then we have $$\kappa <
\psi(x,X) < \chi(x,X)\,.$$ In particular, then $\chi(x,X) \ge
\kappa^{++}\,$.
\end{proposition}

Finally, we recall that a space $X$ is called $\kappa$-bounded if
every subset of $X$ of cardinality $\le \kappa$ has compact
closure in $X$. It is obvious that every $\kappa$-bounded space
$X$ is initially $\kappa$-compact, i.e. every open cover of $X$ of
size at most $\kappa$ has a finite subcover, or equivalently:
every infinite subset of $X$ of cardinality at most $\kappa$ has a
complete accumulation point.

\bigskip

\section{Main results}\label{main}

Our main results make essential use of certain {\em strong
colorings} that were introduced and established by the second
author. Therefore we start with defining these colorings.

\begin{definition}
Let $\lambda$ and $\kappa$ be infinite cardinals. We shall denote
by $\Col(\lambda,\kappa)$ the following statement: There is a
coloring $c: [\lambda]^2 \to 2\,$ such that, given any ordinal
$\xi < \kappa^+$ and a map $h : \xi \times \xi \to 2$, for every
family $\{A_\alpha : \alpha < \lambda\}$ of $\lambda$ many
pairwise disjoint subsets of $\lambda$ of order type $\xi$ there
are $\alpha < \beta < \lambda$ for which
$$c(a_{\alpha,i},a_{\beta,j}) = h(i,j)\,$$
holds for all pairs $\langle i,j \rangle \in \xi \times \xi$.
Here, of course, $a_{\alpha,i}$ denotes the $i$th member of
$A_\alpha$ in its increasing ordering (of type $\xi$).
\end{definition}

Thus our relation $\Col(\lambda,\kappa)$ is identical with the
relation $\pr_0(\lambda,\lambda,2,\kappa^+)$ that was defined by
the second author e.g. in \cite{Sh}, Appendix 1, def. 1.1.

We mention that, simply putting together the results given in
4.6C(5) and 4.5(3) from chapter III of \cite{Sh} (the first result
can be found on page 172 and the second on page 170), one obtains
the following fact.

\begin{proposition}\label{pr:book}
For every infinite cardinal $\kappa$ the relation
$\pr(\lambda,\lambda,\lambda,\kappa^+)$ (that is stronger than
$\pr_0(\lambda,\lambda,2,\kappa^+)\, \equiv\,
\Col(\lambda,\kappa)$) holds for the cardinal $$\lambda =
(2^\kappa)^{++} + \omega_4\,.$$
\end{proposition}

We note that $(2^\kappa)^{++} < \omega_4$ can only occur if
$\kappa = \omega$ and the continuum hypothesis holds, hence in
every other case we have $\Col((2^\kappa)^{++},\kappa)$.

Motivated by the work on this paper, the second author has
achieved some further improvements on this proposition that will
appear in \cite{Sh2}.

For every coloring $c: [\lambda]^2 \to 2\,$ one can naturally
define a subspace $F[c]$ of the Cantor cube $2^\lambda$ as
follows: $F[c] = \{ c_\alpha : \alpha < \lambda \}\,$ where, for
any $\alpha < \lambda$, the point $c_\alpha \in 2^\lambda$ is
defined by the stipulation

$$c_\alpha(\beta) =\left\{
\begin{array}{lll}
c(\alpha,\beta)& \mbox{ if $\,\,\beta < \alpha\,,$}\\
1 & \mbox{ if $\,\,\beta = \alpha\,,$}\\
0 & \mbox{ if $\,\,\beta > \alpha\,$}.
\end{array}
\right.
$$

\medskip

Here, and in what follows, we commited the innocent abuse of
notation of writing $c(\alpha,\beta)$ instead of
$c(\{\alpha,\beta\})$. The requirement $c_\alpha(\alpha) = 1$ is
purely technical, just to ensure that $\alpha \ne \beta$ implies
$c_\alpha \ne c_\beta$.

We shall need the following lemma in the proof of our main result.

\begin{lemma}\label{lm:dense}
If $\lambda$ is an uncountable regular cardinal and
$\,\Col(\lambda,\kappa)$ holds then there is a coloring $d :
[\lambda]^2 \to 2\,$ establishing $\,\Col(\lambda,\kappa)$ with
the extra property that the set $F[d]$ is dense in the Cantor cube
$2^\lambda$.
\end{lemma}

\begin{proof}
Assume that the coloring $c: [\lambda]^2 \to 2\,$ witnesses
$\,\Col(\lambda,\kappa)$. It is obvious that then for each $\alpha
< \lambda$ the coloring $c \upharpoonright [\lambda \setm
\alpha]^2$, i.e. $c$ restricted to the pairs from a tail of
$\lambda$, when "translated" back to $\lambda$ is also a witness
for $\,\Col(\lambda,\kappa)$. This translated coloring
$\,c^{(\alpha)} : [\lambda]^2 \to 2\,$ is naturally defined by the
formula
$$c^{(\alpha)}(\xi,\zeta) = c(\alpha\dotplus\xi,\alpha\dotplus\zeta).$$
Here we use $\dotplus$ to denote ordinal addition.

Next we show that there is an $\alpha < \lambda$ for which $F[c]$
is $\lambda$-dense in the tail product $2^{\lambda \setm \alpha}.$
This means that for every finite function $\,\varepsilon \in
Fn(\lambda \setm \alpha,2)\,$ we have $|[\varepsilon] \cap F[c]| =
\lambda$, where $[\varepsilon] = \{ f \in 2^\lambda : \varepsilon
\subs f \}$ is the elementary open set in the Cantor cube
$2^\lambda$ coded by $\varepsilon$.

Assume, arguing indirectly, that there is no such $\alpha <
\lambda$. We may then define a $\lambda$-sequence
$\langle\varepsilon_\alpha : \alpha < \lambda\rangle$ of members
of $Fn(\lambda,2)$ as follows. Assume that $\alpha < \lambda$ and
$\langle\varepsilon_\beta : \beta < \alpha\rangle$ have already
been defined in such a way that for each $\beta < \alpha$ we have
$|[\varepsilon_\beta] \cap F[c]| < \lambda$. For $\beta < \alpha$
we shall write $E_\beta = \{i : c_i \in [\varepsilon_\beta]\}$,
then $|E_\beta| = |[\varepsilon_\beta] \cap F[c]| < \lambda$.

Since $\lambda$ is regular, we may then find an ordinal
$\nu_\alpha < \lambda$ such that $$\bigcup
\{\dom(\varepsilon_\beta) \cup E_\beta  : \beta < \alpha \} \subs
\nu_\alpha\,.$$ We then choose $\varepsilon_\alpha \in Fn(\lambda
\setm \nu_\alpha)$ so that $|[\varepsilon_\alpha] \cap F[c]| <
\lambda$. This is possible by our indirect assumption.

By $\lambda = \cf(\lambda) > \omega$, after an appropriate
thinning out the sequence $\langle\varepsilon_\alpha : \alpha <
\lambda\rangle$ may be assumed to be such that there are a
positive natural number $n$ and a function $\varepsilon : n \to 2$
for which we have $\varepsilon_\alpha = \varepsilon *
\dom(\varepsilon_\alpha )$ for all $\alpha < \lambda$. The latter
equality means that $|\dom(\varepsilon_\alpha)| = n$ and
$\varepsilon_\alpha(\xi_{\alpha,k}) = \varepsilon(k)$ for each $k
< n$, where, of course, $\xi_{\alpha,k}$ denotes the $k$th element
of $\dom(\varepsilon_\alpha)$ in its increasing order.

Now, let $h : n \times n \to 2$ be any map with the property that
$h(0,k) = \varepsilon(k)$ for all $k < n$. Since $c$ witnesses
$\,\Col(\lambda,\kappa)$, we may then find $\beta < \alpha <
\lambda$ such that for each $k < n$ we have
$$c_{\xi_{\alpha,0}}(\xi_{\beta,k}) = c(\xi_{\alpha,0},\xi_{\beta,k}) = h(0,k) = \varepsilon(k) =
\varepsilon_\beta(\xi_{\beta,k})\,.$$ In other words, this means
that $c_{\xi_{\alpha,0}} \in [\varepsilon_\beta]$, i.e.
$\xi_{\alpha,0} \in E_\beta$ which is a contradiction as $E_\beta
\subs \nu_\alpha$ while $\xi_{\alpha,0} \in
\dom(\varepsilon_\alpha) \subs \lambda \setm \nu_\alpha$.

So fix $\alpha < \lambda$ for which $F[c]$ is $\lambda$-dense in
$2^{\lambda \setm \alpha}.$ We claim that then $F[c^{(\alpha)}]$
is dense, even $\lambda$-dense, in $2^\lambda$. To see this,
consider any $\varepsilon \in Fn(\lambda,2)$ and define
$\widehat{\varepsilon} \in Fn(\lambda \setm \alpha)$ as the
natural translate of $\varepsilon$ by $\alpha$. In other words,
$\dom(\widehat{\varepsilon}) = \{\alpha \dotplus \xi : \xi \in
\dom(\varepsilon)\}$ and $\widehat{\varepsilon}(\alpha \dotplus
\xi) = \varepsilon(\xi)$ for each $\xi \in \dom(\varepsilon)$.

Then $|[\widehat{\varepsilon}] \cap F[c]| = \lambda$, hence the
set $$E = \{\zeta : c_{\alpha \dotplus \zeta} \in
[\widehat{\varepsilon}] \mbox{ and } \dom(\widehat{\varepsilon})
\subs \alpha \dotplus \zeta\}$$ is also of cardinality $\lambda$.
But this means that for every $\zeta \in E$ and $\xi \in
\dom(\varepsilon)$ we have $$\big({c^{(\alpha)}}\big)_\zeta(\xi) =
c_{\alpha \dotplus \zeta}(\alpha \dotplus \xi) =
\widehat{\varepsilon}(\alpha \dotplus \xi) = \varepsilon(\xi)\,,$$
hence $\big({c^{(\alpha)}}\big)_\zeta \in [\varepsilon]$ holds for
each $\zeta \in E$. But this clearly implies $|[\varepsilon] \cap
F[c^{(\alpha)}]| = \lambda$, showing that $d = c^{(\alpha)}$
satisfies the requirements of the lemma.

\end{proof}

We are now ready to state and prove our main result. Before
formulating it, however we recall that the $\kappa$-closure
$cl_\kappa(A)$ of a subset $A$ of a topological space $X$ is
defined by
$$cl_\kappa(A) = \bigcup \{\overline{B} : B \in [A]^{\le\kappa}\},$$
where $\overline{B}$ denotes the closure of $B$ in $X$. Moreover,
for every point $x \in 2^\lambda\,$ its support $\supp(x)$ is
defined by
$$\supp(x) = \{\alpha < \lambda : x(\alpha) = 1\}\,.$$
\begin{theorem}\label{tm:main}
Assume that $\kappa$ is an infinite and $\lambda > \kappa^+$ is a
regular cardinal, moreover the coloring $c: [\lambda]^2 \to 2\,$
witnesses the relation $\,\Col(\lambda,\kappa)$. Let us denote by
$H_\kappa[c]$ the $\kappa$-closure of the set $F[c]$ in the Cantor
cube $2^\lambda$. Then for every right separated subset $S$ of
$H_\kappa[c]$ there is an $\alpha < \lambda$ for which
$$\bigcup \{\supp(x) : x \in S\} \subs \alpha\,.$$
\end{theorem}

\begin{proof}
Assume that the statement of the theorem fails. Then we clearly
may find a subset $S = \{x_\alpha : \alpha < \lambda\} \subs
H_\kappa[c]$ that is right separated by the well-ordering given by
the indices $\alpha$, moreover for each $\alpha < \lambda$ we have
$\,\supp(x_\alpha) \setm \alpha \ne \emptyset$. The first part of
this means that for every $\alpha < \lambda$ there is a finite set
$a_\alpha \in [\lambda]^{<\omega}$ such that $\varepsilon_\alpha =
x_\alpha \upharpoonright a_\alpha$ codes an elementary open right
separating neighbourhood of the point $x_\alpha$ in $S$, i.e.
$x_\beta \notin [\varepsilon_\alpha ]$ for all $\beta > \alpha$.
The second part implies that for every $\alpha < \lambda$ there is
a set $A_\alpha \in [\lambda \setm \alpha]^{\kappa}$ such that
$x_\alpha \in \overline{\{ c_i : i \in A_\alpha \}}$.

A standard delta-system and counting argument allows us to thin
out the sequence $\{x_\alpha : \alpha < \lambda\}$ in such a way
that the sets $a_\alpha$ are pairwise disjoint and of the same
size $n$. Moreover, similarly as above in the proof of lemma
\ref{lm:dense}, we may in addition assume that for some
$\varepsilon \in 2^n$ we have $\varepsilon_\alpha = \varepsilon *
a_\alpha$ for each $\alpha < \lambda$. Let us now set $B_\alpha =
a_\alpha \cup A_\alpha$ for $\alpha < \lambda$. After some further
thinning out, using $\lambda = \cf(\lambda) > \kappa^+$, we may
also assume that all the sets $B_\alpha$ have the same order type
$\xi < \kappa^+$, moreover $\sup B_\beta < \min B_\alpha$ whenever
$\beta< \alpha < \lambda$. Finally, we may also assume that there
is some fixed set $a \in [\xi]^n$ so that for each $\alpha <
\lambda$ we have $a_\alpha = \{\zeta_{\alpha,\nu} : \nu \in a \}$,
where $B_\alpha = \{\zeta_{\alpha,\nu} : \nu \in \xi \}$ is the
increasing enumeration of $B_\alpha$.

Now let $h : \xi \times \xi \to 2$ be any map satisfying
$h(\eta,\nu_k) = \varepsilon(k)$ for all $\eta < \xi$ and $k < n$,
where $\nu_k$ is the $k$th member of the set $a \in [\xi]^n$ in
its increasing order. Since the coloring $c$ witnesses the
relation $\,\Col(\lambda,\kappa)$ we may then find $\beta < \alpha
< \lambda$ such that $h(\nu,\mu) =
c(\zeta_{\alpha,\nu},\zeta_{\beta,\mu})$ holds for each pair
$\langle \nu,\mu \rangle \in \xi \times \xi$. But according to our
above arrangements this implies $c_i \in [\varepsilon_\beta]$ for
each $i \in B_\alpha \supset A_\alpha$, consequently $x_\alpha \in
[\varepsilon_\beta]$ as well because $[\varepsilon_\beta]$ is a
closed (in fact clopen) set and $x_\alpha \in \overline{\{ c_i : i
\in A_\alpha \}}$. This, however, is a contradiction because
$[\varepsilon_\beta]$ was assumed to be a right separating
neighbourhood of $x_\beta$ which thus cannot contain the point
$x_\alpha$. This contradiction then completes the proof of theorem
\ref{tm:main}.

\end{proof}

In what follows, let us denote by $\Sigma_\lambda$ the subset of
the the Cantor cube $2^\lambda$ that consists of all points $x \in
2^\lambda$ whose support is bounded in $\lambda$. (Of course, if
$\lambda$ is regular this is equivalent with $|\supp(x)| <
\lambda$.) Using this notation, for every coloring $c: [\lambda]^2
\to 2\,$ we have, by definition, $F[c] \subs \Sigma_\lambda$.
Moreover, if $c$ witnesses the relation $\,\Col(\lambda,\kappa)$
and $\lambda > \kappa^+$ is a regular cardinal then, by theorem
\ref{tm:main}, we even have $\overline{S} \subs \Sigma_\lambda$
whenever $S \subs H_\kappa[c]$ is right separated. Thus we have
arrived at the following result that makes the statement made in
the title of our paper precise.

\begin{corollary}\label{co:init}
If $\lambda > \kappa^+$ is a regular cardinal and
$\,\Col(\lambda,\kappa)$ holds then there is a dense
$\kappa$-bounded subspace of the the Cantor cube $2^\lambda$ that
has a discretely untouchable (non-isolated) point.
\end{corollary}

\begin{proof}
By lemma \ref{lm:dense} there is a coloring $c: [\lambda]^2 \to
2\,$ witnessing $\,\Col(\lambda,\kappa)$ for which $F[c]$ is dense
in $2^\lambda$. Now pick any point $x \in 2^\lambda \setm
\Sigma_\lambda$, i.e. with $|\supp(x)| = \lambda$ and set $X =
H_\kappa[c] \cup \{x\}$. Then $H_\kappa[c]$ is $\kappa$-bounded
being the $\kappa$-closure of a subset of the compact space
$2^\lambda$, hence so is $X$. Moreover, $x$ is an accumulation
point of $H_\kappa[c]$, as already $F[c] \subs H_\kappa[c]$ is
dense in $2^\lambda$. But by theorem \ref{tm:main} no discrete (or
equivalently: right separated) subset of $H_\kappa[c]$ has $x$ in
its closure.

\end{proof}

For each infinite cardinal $\kappa$, according to proposition
\ref{pr:book} from the beginning of this section, $\,\lambda =
(2^\kappa)^{++} + \omega_4$ satisfies the assumption of corollary
\ref{co:init}. In particular, for $\kappa = \omega$ the smallest
value we get for such a $\lambda$ is $\omega_4$, provided that the
continuum is $\le \omega_2$.

Note that in corollary \ref{co:init} the character $\chi(x,X)$ of
the discretely untouchable point $x \in X$ is $\lambda$. On the
other hand, proposition \ref{pr:init} yields the lower bound
$\kappa^{++}$ for the character of a discretely untouchable
non-isolated point in an initially $\kappa$-compact space. So it
is natural to raise the question if the value for the character of
such a point could be lower than $(2^\kappa)^{++} + \omega_4$. The
following problem seems to be the most intriguing.

\begin{problem}\label{pm:w_2}
Is it consistent with (or even provable from) ZFC that there is a
discretely untouchable non-isolated point of character $\omega_2$
(or $\omega_3$) in some countably compact (or $\omega$-bounded)
regular space?
\end{problem}

It is standard to show that the cardinality of the
$\kappa$-bounded space $X$ given in corollary \ref{co:init} is
$|H_\kappa[c]| = |cl_\kappa \big(F([c]\big)| = \lambda^\kappa
\cdot 2^{2^ \kappa}$. However, if instead of $\kappa$-bounded we
only want an initially $\kappa$-compact example, then this value
may be chosen to be just $\lambda^\kappa$. Indeed, this can be
achieved by constructing a subspace $Y \subs X$ that includes
$F[c] \cup \{x\}$ and has the property that every infinite set $A
\in [Y]^{\le \kappa}$ has a complete accumulation point in $Y$.

In particular, for $\kappa = \omega$ this yields us a countably
compact regular space with a discretely untouchable non-isolated
point of cardinality $\omega_4$, provided that the continuum is
$\le \omega_2$. Again, it is an intriguing problem if the
cardinality of such an example can be, consistently, lowered. Let
us note that, since every non-isolated point in a scattered space
is discretely touchable, such an example cannot be scattered and
hence must be of cardinality at least continuum.

\begin{problem}\label{pm:w_1}
Is it consistent with ZFC that there is a countably compact (or
$\omega$-bounded) regular space of cardinality $\omega_1$ (or
$\omega_2$, or $\omega_3$) with a discretely untouchable
non-isolated point?
\end{problem}

\bigskip

\section{Countable examples}\label{ctbl}

We have shown in the previous section that for every cardinal
$\kappa$ there is a $\kappa$-bounded, and hence initially
$\kappa$-compact, regular space with a discretely untouchable
non-isolated point. In the introduction we also promised to
exhibit such points in Lindelöf regular spaces to conclude that
basically no weakening of compactness suffices to preserve the
property of compact Hausdorff spaces that was our starting point.

In fact, we would like to point out that there are even countable,
hence hereditarily Lindelöf regular spaces that are crowded, i.e.
have no isolated points, in which all discrete subsets are closed,
hence all points are discretely untouchable. Perhaps the first
such example, a countable maximal space that is regular, was
constructed by E. van Douwen; his example was published in the
postumus paper \cite{vD}. A very different such example is the
countable submaximal dense subspace of the Cantor cube
$2^{2^\omega}$ that was constructed in theorem 4.1 of \cite{JSSz}.

Both of these examples are rather non-trivial, so we decided to
include in this paper the following result which shows that
actually every crowded irresolvable countable regular space
contains such an open subspace.

We recall that a space is called irresolvable if it has no two
disjoint dense subsets and that there is a crowded, countable, and
regular irresolvable space. The existence of such a space was
first established by E. Hewitt in 1943, in his classical paper
\cite{H} on resolvability.

\begin{proposition}\label{pr:w}
Every crowded irresolvable regular space has an open subspace in
which all countable discrete subsets are closed.
\end{proposition}

\begin{proof}
Let $X$ be a crowded irresolvable regular space. It is well-known
that every irresolvable space has an open subspace that is
hereditarily irresolvable, so let $Y$ be such an open subspace of
$X$. It suffices to show that the set $Z$ of all accumulation
points of countable discrete sets cannot be dense in $Y$ because
then the interior of $Y \setm Z$ is the required open set.

Assume, on the contrary that $Z$ is dense in $Y$. Then $Y \setm Z$
cannot be dense in $Y$ because $Y$ is irresolvable. This means
that $U = Int(Z) \ne \emptyset$. But then, by definition, every
point of $U$ is the accumulation point of a countable discrete
subset of U. Now, every countable discrete set in a regular space
is strongly discrete, i.e. its points can be separated by pairwise
disjoint open sets. But by theorem 2.1 of \cite{Sh-Sh}, (see also
theorem 1.3 of \cite{JSSz2}), then $U$ is even
$\omega$-resolvable, i.e. it has infinitely many pairwise disjoint
dense subsets, which contradicts the choice of $Y$.

\end{proof}

\end{document}